\documentclass[12pt,a4paper,leqno]{article}

\usepackage[utf8]{inputenc}
\usepackage[T1]{fontenc}
\usepackage[english]{babel}
\usepackage{amsthm}
\usepackage{amsfonts}         
\usepackage{amsmath}
\usepackage{amssymb}
\usepackage{hyperref}
\usepackage{amscd}
\usepackage{enumerate}
\usepackage{tikz}
\usepackage{mathrsfs}  
\usetikzlibrary{matrix, arrows, decorations.pathmorphing}
\usepackage{hyperref}

\newcommand{\fref}[1]{\hyperref[{#1}]{\ref*{#1}}}

\newcommand{\Ab}{\mathbb{A}}

\newcommand{\Eb}{\mathbb{E}}

\newcommand{\Nb}{\mathbb{N}}

\newcommand{\Pb}{\mathbb{P}}

\newcommand{\Zb}{\mathbb{Z}}

\newcommand{\Cc}{\mathcal{C}}

\newcommand{\Fc}{\mathcal{F}}
\newcommand{\Gc}{\mathcal{G}}
\newcommand{\Hc}{\mathcal{H}}

\newcommand{\Kc}{\mathcal{K}}

\newcommand{\Mc}{\mathcal{M}}

\newcommand{\Oc}{\mathcal{O}}

\newcommand{\Sc}{\mathcal{S}}

\newcommand{\Zc}{\mathcal{Z}}

\newcommand{\Ls}{\mathscr{L}}

\newcommand{\Proj}{\mathbb{P}}

\newcommand{\OO}{\mathcal{O}}

\newcommand{\Li}{\mathscr{L}}

\newcommand{\coim}{\mathrm{\coim}}

\newcommand{\Spec}{\mathrm{Spec}}

\newcommand{\QCoh}{\mathrm{QCoh}}
\newcommand{\Perf}{\mathrm{Perf}}

\newtheorem{theo}{Tplottin ubuntuheorem}[section]
\theoremstyle{plain}
\newtheorem{thm}[theo]{Theorem}
\newtheorem{lem}[theo]{Lemma}
\newtheorem{prop}[theo]{Proposition}
\newtheorem{cor}[theo]{Corollary}
\newtheorem*{thm*}{Theorem}
\newtheorem*{lem*}{Lemma}
\newtheorem*{prop*}{Proposition}
\newtheorem*{cor*}{Corollary}

\theoremstyle{definition}
\newtheorem{defn}[theo]{Definition}

\newtheorem{ex}[theo]{Example}

\theoremstyle{remark}
\newtheorem{rem}[theo]{Remark}

\title{Ample line bundles, global generation and $K^0$ on quasi-projective derived schemes}
\author{Toni Annala}
\date{}
\begin{document}

\maketitle
\begin{abstract}
The purpose of this article is to extend some classical results on quasi-projective schemes to the setting of derived algebraic geometry. Namely, we want to show that any vector bundle on a derived scheme admitting an ample line bundle can be twisted to be globally generated. Moreover, we provide a presentation of $K^0(X)$ as the Grothendieck group of vector bundles modulo exact sequences on any quasi-projective derived scheme $X$.
\end{abstract}

\tableofcontents

\section{Introduction}

The purpose of this paper is to give proofs for some general facts concerning quasi-projective derived schemes stated and used in \cite{An}. In that paper, the author constructed a new extension of algebraic cobordism from smooth quasi-projective schemes to all quasi-projective (derived) schemes. The main result of the paper is that the newly obtained extension specializes to $K^0(X)$ -- the $0^{th}$ algebraic $K$-theory of the scheme $X$, giving strong evidence that the new extension is the correct one. All previous extensions (operational, $\Ab^1$-homotopic) fail to have this relation with the Grothendieck ring $K^0$.

A crucial role in the proof is played by the construction of Chern classes for $\Omega^*$. The construction used the following key properties that are known to hold classically
\begin{enumerate}
\item Any vector bundle on a quasi-projective scheme can be made globally generated by tensoring it with a line bundle.

\item The finite Grassmannians represent the functor assigning to each $X$ the set of tuples $(E,s_1,...,s_n)$ of a vector bundle of rank $r$ and a sequence $s_1,...,s_n$ of global sections that generate.
\end{enumerate}
In addition, we needed the following result
\begin{enumerate}
\item[3.] If $X$ is quasi-projective, then the group $K^0(X)$ is the Grothendieck group of vector bundles, i.e., it is the Abelian group generated by equivalence classes of vector bundles modulo relations coming from short exact sequences.
\end{enumerate}
The purpose of this paper is to extend the facts 1 (Corollary \fref{GlobalGeneration}) and 3 (Theorem \fref{KTheoryTheorem}) to their natural analogues for derived schemes; fact 2 was already extended in \cite{An} (essentially already in \cite{Lur3}). The results, while being direct analogues of classical results, are of great practical value when one wants to work geometrically with derived algebraic geometry.

\section{Background}

We now recall some necessary background material.

\subsection{$\infty$-Categories}

In accordance with most of the modern work on derived algebraic geometry, we will cast our results in the language of $\infty$-categories. By an $\infty$-category, we mean more precisely an $(\infty,1)$-category, i.e., all higher homotopies are invertible. There are multiple different models for the theory of $\infty$-categories, and the choice is not very important for the results of this paper, but as our main references will be \cite{Lur1}, \cite{Lur2} and \cite{Lur3}, we are going to settle with the language of quasi-categories.

In the theory of categories, the category of sets plays a distinguished role. For example, the morphisms between two objects in a category forms a set, Yoneda-embedding realizes any category as a subcategory of functors taking values in sets etc. In the theory of $\infty$-categories, a similar role is played by the $\infty$-category of spaces. We now review some more specific details from the theory.

Throughout this paper, we are going to call $\infty$-categorical (co)limits \emph{homotopy (co)limits}. This is done mainly in order to avoid confusion when working with objects for which both concepts would make sense.  Similarly, we will call the $\infty$-categorical analogues of sheaves \emph{stacks}. On a topological space $X$, a stack $\Fc$ is a presheaf taking values in some $\infty$-category satisfying homotopical analogue of the sheaf condition. Essentially this boils down to requiring the square
\begin{center}
\begin{tikzpicture}[scale = 1.5]
\node (A2) at (2,2) {$\Gamma(U_1; \Fc)$};
\node (B2) at (2,1) {$\Gamma(U_1 \cap U_2; \Fc)$};
\node (A1) at (0,2) {$\Gamma(U_1 \cup U_2; \Fc)$};
\node (B1) at (0,1) {$\Gamma(U_2; \Fc)$};
\path[every node/.style={font=\sffamily\small}]
(A1) edge[->] (A2)
(B1) edge[->] (B2)
(A1) edge[->] (B1)
(A2) edge[->] (B2)
;
\end{tikzpicture}
\end{center}
to be homotopy Cartesian for any two open sets $U_1, U_2$. 

Suppose $\Cc$ is a pointed $\infty$ category. For a morphism $X \to Y$, we define the \emph{homotopy cofibre} $C$ as the homotopy pushout
\begin{center}
\begin{tikzpicture}[scale = 1.5]
\node (A2) at (1,2) {$Y$};
\node (B2) at (1,1) {$C$};
\node (A1) at (0,2) {$X$};
\node (B1) at (0,1) {$0$};
\path[every node/.style={font=\sffamily\small}]
(A1) edge[->] (A2)
(B1) edge[->] (B2)
(A1) edge[->] (B1)
(A2) edge[->] (B2)
;
\end{tikzpicture}
\end{center}
Dually, one may define the \emph{homotopy fibre} of a morphism as a homotopy pullback. A sequence $X \to Y \to Z$ is called a \emph{homotopy cofibre sequence} if $Z$ is a model to the homotopy cofibre of $X \to Y$. Similarly, such a sequence is a \emph{homotopy fibre sequence} if $X$ is a model for the homotopy fibre of $Y \to Z$.  
\begin{defn}
A pointed $\infty$-category $\Cc$ is called \emph{stable} if it has all homotopy (co)fibres, and if the homotopy fibre and cofibre sequences of $\Cc$ coincide.
\end{defn}
The notion of a stable $\infty$-category is an enhancement of the notion of a triangulated category in the sense that the homotopy category naturally admits a structure of a triangulated category. A direct consequence of the definition is that homotopy cofibre sequences are preserved in homotopy pushouts and homotopy pullbacks.

\subsection{Derived Schemes}

While there are many equivalent ways to define the notion of a derived scheme, the following (cf. \cite{Lur3} Definition 1.1.2.8., \cite{KST} Definition 2.1) seems to be best for the purposes of this paper.

\begin{defn}
A \emph{derived scheme} $X$ is a topological space $X_\mathrm{top}$ together with a hypercomplete stack $\Oc_X$ of simplicial commutative rings. We moreover assume that 
\begin{enumerate}
\item[(i)] The \emph{truncation} $tX := (X, \pi_0 \Oc_X)$ is a scheme,
\item[(ii)] the sheaves $\pi_n\Oc_X$ are quasi-coherent over $\pi_0 \Oc_X$.
\end{enumerate}
\end{defn} 
In this article we will only deal with schemes with a Noetherian and finite dimensional underlying space, and hence the hypercompleteness assumption is automatically satisfied. Moreover, we will work over a field $k$ of characteristic 0, where derived schemes may equivalently be characterized as topological spaces equipped with a stack of commutative and connective differential graded $k$-algebras satisfying analogues of the conditions (i) and (ii) above. In the sequel we will often talk about \emph{derived rings}, by which we will mean either a connective dg-algebra or a simplicial ring over $k$ (or even connective $\Eb_\infty$-ring spectra over $k$).

Any commutative derived ring defines an \emph{affine derived scheme} $\Spec(A)$. The full sub $\infty$-category of affine derived schemes is equivalent to the opposite $\infty$-category of derived rings. It is known that a derived scheme $X$ is equivalent to an affine one if and only if the truncation $(X_\mathrm{top}, \pi_0 \Oc_X)$ is isomorphic to an affine scheme.

\begin{defn}
A $(1-)$morphism $f: X \to Y$ between derived schemes is a continuous map $f: X_\mathrm{top} \to Y_\mathrm{top}$ of topological spaces and a map $f^\sharp:\Oc_Y \to f_* \Oc_X$ of stacks so that $(f, \pi_0 f^\sharp)$ defines a map of schemes.

A morphism $f: X \to Y$ is \emph{affine} if for all maps $\Spec(A) \to Y$ from an affine derived scheme, the homotopy pullback $X \times^h_Y \Spec(A)$ is affine. The morphism is moreover a \emph{closed embedding} if the induced map $\Spec(B) \to \Spec(A)$ determines a surjective map $\pi_0 A \to \pi_0 B$ of rings (cf. \cite{TV2} Definition 2.2.3.5. (2)).
\end{defn}

\begin{rem}
It is clear that a map $X \to Y$ between derived schemes is a closed embedding exactly when the induced map between truncations is a closed embedding of schemes.
\end{rem}

\begin{defn}
A morphism $X \to Y$ of derived $k$-schemes is \emph{quasi-projective} if it factors through a closed embedding $X \hookrightarrow U \times Y$, where $U$ is an open subscheme of some projective scheme $\Proj^n$. A derived $k$-scheme $X$ is \emph{quasi-projective} if the structure morphism $X \to pt$ is. 
\end{defn}

\begin{rem}
We note that any quasi-projective derived scheme $X$ is \emph{separated}, i.e., the diagonal map $X \to X \times_k X$ is a closed embedding. A standard consequence is that the intersection of finite number of affine opens is again an affine open. The underlying topological space $X$ is always compact, so $X$ is \emph{quasi-compact}.
\end{rem}

\begin{defn}
We say that a derived scheme $X$ is \emph{Noetherian} (cf. \cite{KST} Definition 2.4) if the truncation $tX$ is Noetherian, and if the homotopy sheaves $\pi_i(\Oc_X)$ are finite type over $\pi_0(\Oc_X)$. Note that for quasi-projective derived schemes the first assumption is automatically satisfied, so for them the definition really just boils down to a finiteness condition on the structure stack $\Oc_X$.
\end{defn}

\subsection{Quasi-Coherent Sheaves}

As any derived scheme $X$ has the underlying topological space $X_\mathrm{top}$, it makes sense of talking about stacks on $X$ other than the structure stack $\Oc_X$. An important class of such objects is given by the quasi-coherent sheaves on $X$. For an affine scheme $\Spec(A)$, the stable $\infty$-category $\QCoh(\Spec(A))$ of quasi-coherent sheaves is equivalent to the $\infty$-category of (unbounded) $A$-dg-modules. Moreover, $\QCoh$ defines a stack of stable $\infty$-categories on the $\infty$-category of derived schemes (with respect to faithfully flat topology), which means, among other things, that $\QCoh(X)$ of a separated derived scheme $X$ is obtained by gluing the module categories associated to some affine open cover of $X$.

For the practical purposes of the paper, it is necessarily only to note that a quasi-coherent sheaf $\Fc$ on $X$ is a stack of unbounded chain-complexes of Abelian groups which, when restricted to an affine open subset, is completely determined by a single dg-module. As homotopy limits of dg-modules can be computed as homotopy limits of Abelian chain complexes (or even as homotopy limits of spectra), and as we will mainly be interested in understanding global sections in terms of data, any structure makes little or no difference for us. Like in the classical case, we often want to deal with quasi-coherent sheaves with certain finiteness properties.

\begin{defn}
Suppose $X$ is a quasi-projective derived scheme. We say that a quasi-coherent sheaf $\Fc$ is \emph{coherent} if it is eventually connective (meaning that $\pi_i(\Fc) = 0$ for $i \ll 0$) and the homotopy sheaves $\pi_i \Fc$ are coherent over the truncation $\pi_0 X$. 
\end{defn}
\begin{rem}
As quasi-projective derived schemes satisfy certain finiteness conditions, the above definition coincides with the definition of \emph{almost perfect} sheaves in \cite{Lur2} and \cite{Lur3}.
\end{rem}

Another special class of quasi-coherent sheaves, used in the definition of algebraic $K$-theory, are the perfect sheaves. If $A$ is a derived ring, then an $A$-module $M$ is called \emph{perfect} if it is a retract of a finite homotopy colimit of shifts of $A$ (equivalently $M$ is a compact module). A quasi-coherent sheaf $\Fc$ is \emph{perfect} if the restriction $\Fc \vert _U$ is perfect module for all affine opens $U \subset X$ (or equivalently, for all affine opens in a cover). As our derived rings are assumed to be connective, it follows that all perfect sheaves are eventually connective. For nice enough derived schemes, perfect sheaves have a nice alternative description:

\begin{prop}[cf. \cite{Lur3} Proposition 9.6.1.1.]
If $X$ is a quasi-compact and quasi-separated derived scheme (e.g., a quasi-projective derived scheme), then the perfect sheaves are exactly the compact objects of $\QCoh(X)$. Moreover, $\QCoh(X)$ is compactly generated.
\end{prop}

In the statement above, a \emph{compact object} $C$ is such that the mapping space functor $\hom (C, -)$ commutes with filtered homotopy colimits. A stable $\infty$-category is said to be \emph{compactly generated} if there exists a set $\Sc$ of compact objects whose right orthogonal vanishes, i.e., $\hom(C, M) \simeq pt$ for all $C \in \Sc$ implies that $M \simeq 0$.

\section{Strong Sheaves}

The purpose of this section is study a nicely behaved subclass of quasi-coherent sheaves: the so called \emph{strong sheaves}. The topology of such sheaves is in some sense completely determined by the topology on the structure stack $\Oc_X$, making them easier to understand. We also give a nice characterization of cofibre sequences of strong sheaves.

Recall from \cite{TV1} that a simplicial module $M$ over a simplicial commutative ring $A$ is \emph{strong} if the natural map of graded rings
\begin{equation*}
\pi_*(A) \otimes_{\pi_0(A)} \pi_0(M) \to \pi_*(M)
\end{equation*}
is an isomorphism. A simplicial $A$-algebra $B$ is \emph{strong} if it is strong considered as an $A$-module. We claim that being strong is a local property:

\begin{prop}
Let $A \to \widetilde A$ be finite étale cover of simplicial rings. Now an $A$-module $M$ is strong if and only if $\widetilde M := \widetilde A \otimes_A^L M$ is strong over $\widetilde A$.
\end{prop}
\begin{proof}
Recall that by the remark following Definition 2.2.2.12 in \cite{TV2}, $A \to \widetilde A$ is a finite étale cover if and only if $A \to \widetilde A$ is strong and the truncation morphism is a finite étale cover in the usual sense. We note that the graded ring $\pi_*(\widetilde A)$ is flat over $\pi_*(A)$, so that the natural map $\pi_*(\widetilde A) \otimes^L_{\pi_*(A)} \pi_*(M) \to \pi_*(\widetilde M)$ is an isomorphism.

Let's consider the natural map $\pi_*(\widetilde A) \otimes_{\pi_0(A)} \pi_0(M) \to \pi_*(\widetilde M)$. As we have the equality $\pi_*(\widetilde A) = \pi_*(A) \otimes_{\pi_0(A)} \pi_0(\widetilde A)$, the above map gives (after canonical equivalences)
\begin{equation*}
\pi_0(\widetilde A) \otimes_{\pi_0(A)} (\pi_*(A) \otimes_{\pi_0(A)} \pi_0(M)) \to \pi_0(\widetilde A) \otimes_{\pi_0(A)} \pi_*(M)
\end{equation*}
which proves the claim as $\pi_0 (A) \to \pi_0(\widetilde A)$ is faithfully flat.
\end{proof}

\begin{defn}
Let $\Fc$ be a quasi-coherent sheaf on a derived scheme $X$. We say that $\Fc$ is \emph{strong} if for all $\Spec(A) \to X$ the restriction $\Fc \vert_{\Spec(A)}$ of $\Fc$ to $\Spec(A)$ corresponds to a strong $A$-module. By the previous proposition, it is enough to check the condition on the affine open sets in any Zariski or étale cover.
\end{defn}

\begin{defn}\label{StrongExact}
Suppose we have a map $\Fc \to \Gc$ of strong sheaves. We say that it is an \emph{epimorphism} (\emph{monomorphism}) if the induced map on all the homotopy sheaves is an epimorphism (monomorphism). As tensor products preserve surjections, being an epi is equivalent to the induced map $\pi_0(\Fc) \to \pi_0(\Gc)$ being a surjection. 

A triangle $\Fc \to \Gc \to \Hc$ of strong sheaves is called a \emph{short exact sequence} if the composition $\Fc \to \Hc$ is nil-homotopic, and it induces short exact sequences 
\begin{equation*}
0 \to \pi_i(\Fc) \to \pi_i(\Gc) \to \pi_i(\Hc) \to 0
\end{equation*}
on homotopy sheaves. 
\end{defn}

\begin{rem}
A map $\Fc \to \Gc$ of strong sheaves is an equivalence if and only if it is a monomorphism and an epimorphism. This is equivalent to the induced map $\pi_0(\Fc) \to \pi_0(\Gc)$ being an isomorphism as tensor products preserve isomorphisms.
\end{rem}

\begin{ex}\label{VectorBundles}
\emph{Vector bundles} form an important class of strong sheaves. We recall from \cite{TV2} Section 2.2.6.1 that a \emph{rank $r$ vector bundle} is a strong sheaf $\Fc$ (in our terminology) such that for all $\Spec(A) \to X$ the $\pi_0(A)$-module $\pi_0(\Fc \vert_{\Spec(A)})$ is projective of rank $r$ for all $A$. The condition is equivalent to $\Fc$ being locally free of rank $r$.

Vector bundles are simple in many respects. First of all, they are closed under pullbacks, and the pullback of $E$ along truncation embedding $tX \hookrightarrow X$ is isomorphic to $\pi_0(E)$. Both claims fail in general (a flatness condition is required for them to hold). If we assume $X$ to have some nice finiteness properties, e.g. $X$ quasi-projective, then we see that a map of vector bundles $E \to F$ is an epimorphism exactly when the induced maps of fibres over the \emph{classical points} of $X$ is an epimorphism of vector spaces over the residue field. By classical points, we mean the points of the truncation $tX$, which canonically map to $X$.

Moreover, $E' \to E \to \overline{E}$ is a short exact sequence of vector bundles if and only if 
\begin{equation*}
0 \to \pi_0(E') \to \pi_0(E) \to \pi_0(\overline{E}) \to 0
\end{equation*}
is. Indeed, locally the short exact sequence is split, so at least all the sequences
\begin{equation*}
0 \to \pi_i(A) \otimes_{\pi_0(A)} \pi_0(E') \to \pi_i(A) \otimes_{\pi_0(A)} \pi_0(E) \to \pi_i(A) \otimes_{\pi_0(A)} \pi_0(\overline{E}) \to 0
\end{equation*}
are exact. Moreover, as the sequence is locally equivalent to
\begin{equation*}
0 \to A^{\oplus n} \to A^{\oplus n} \oplus A^{\oplus m} \to A^{\oplus m} \to 0 
\end{equation*}
where the composition $A^{\oplus n} \to A^{\oplus m}$ \emph{is} the zero morphism, we can conclude that $E \to F$ is \emph{locally} nil-homotopic. This is equivalent to saying that $E \to F$ factors locally trough $0$, which in turn implies that it factors globally trough $0$, proving that $E \to F$ is nil-homotopic. Using similar reasoning, one can conclude that the homotopy fibre of a surjective map between vector bundles is a vector bundle.

The rather simple condition on the zeroth homotopy groups also happens to be equivalent for the triangle to be a cofibre sequence -- a special case of the following proposition. 
\end{ex}

\begin{prop}\label{ExactVsCofibre}
A triangle $\Fc \to \Gc \to \Hc$ of strong sheaves is a homotopy cofibre sequence in $\QCoh(X)$ if and only if it is an exact sequence.
\end{prop}
\begin{proof}
We need to show that the triangle $\Fc \to \Gc \to \Hc$ is equivalent to a legitimate cofibre sequence $\Fc \to \Gc \to \Hc'$ where $\Hc'$ is a homotopy pushout
\begin{center}
\begin{tikzpicture}[scale = 1.5]
\node (A2) at (1,2) {$\Gc$};
\node (B2) at (1,1) {$\Hc'$};
\node (A1) at (0,2) {$\Fc$};
\node (B1) at (0,1) {$0$};
\path[every node/.style={font=\sffamily\small}]
(A1) edge[->] (A2)
(B1) edge[->] (B2)
(A1) edge[->] (B1)
(A2) edge[->] (B2)
;
\end{tikzpicture}
\end{center}
As $\Fc \to \Hc$ factors trough $0$ by assumption, we obtain a morphism 
\begin{center}
\begin{tikzpicture}[scale = 1.5]
\node (A2) at (1,2) {$\Gc$};
\node (B2) at (1,1) {$\Gc$};
\node (C2) at (2,1) {$\Hc$};
\node (A1) at (0,2) {$\Fc$};
\node (B1) at (0,1) {$\Fc$};
\node (C1) at (2,2) {$\Hc'$};
\path[every node/.style={font=\sffamily\small}]
(A1) edge[->] (A2)
(B1) edge[->] (B2)
(C1) edge[->] (C2)
(A1) edge[->] (B1)
(A2) edge[->] (B2)
(A2) edge[->] (C1)
(B2) edge[->] (C2)
;
\end{tikzpicture}
\end{center}
of triangles. As $\Fc \to \Gc$ induces injections on all homotopy groups, the cofibre long exact sequence of homotopy groups splits into multiple short exact sequences. We can now conclude using 5-lemma that $\Hc' \to \Hc$ is an equivalence, proving the claim.
\end{proof}

Finally, a notion of crucial importance for us:

\begin{defn}
A strong sheaf $\Fc$ is said to be \emph{globally generated} if it admits an epimorphism
\begin{equation*}
\Oc_X^{\oplus I} \to \Fc
\end{equation*} 
from a free sheaf (not necessarily of finite rank).
\end{defn}

\begin{rem}
For coherent strong sheaves on quasi-projective derived schemes this is equivalent to requiring an epimorphism
\begin{equation*}
\Oc_X^{\oplus n} \to \Fc
\end{equation*} 
for $n \in \Nb$.
\end{rem}

The space of maps $\Oc_X \to \Fc$ is equivalent to the space of global sections of $\Fc$, and hence our definition of global generation really is an enhancement of the classical notion of having enough global sections. Moreover, global sections of $\Fc$ truncate to give global sections of $\pi_0(\Fc)$, and the truncated sections generating $\pi_0(\Fc)$ is equivalent to the original sections generating $\Fc$. Hence the question of whether or not a strong sheaf $\Fc$ is globally generated is almost completely reduced to the same question concerning $\pi_0(\Fc)$. However, there is a caveat: what exactly is the relationship between the global sections of $\Fc$ and the global sections of $\pi_0(\Fc)$? In essence, this is the topic of the following section.

\section{Ample Line Bundles}

This section has some errors: most importantly, Theorem \ref{TwistingTheorem} is true only for derived schemes whose truncation is projective (rather than quasi-projective). All the useful results can be proved by other means (in particular, replacing the definition of ampleness, Definition \ref{Ampleness}, by a better definition). Please consult Section 2 of \cite{An2}, for the details.

It is a classical fact in algebraic geometry that for a quasi-projective scheme $X$ and a coherent sheaf $\Fc$ on $X$, the twists $\Fc(n)$ are globally generated for $n \gg 0$. The main purpose of this section is to extend this result to strong coherent sheaves on a quasi-projective derived scheme $X$. We note that any quasi-projective derived scheme has a natural candidate for the twisting sheaf $\Oc_X(1)$: choose an embedding $X \hookrightarrow \Pb^n$ and pull back $\Oc_{\Pb^n}(1)$.

\subsection{The Descent Spectral Sequence}

Before going any further, we need to understand the space of global sections. Recall that a quasi-coherent sheaf $\Fc$ on a derived scheme $X$ is a stack of chain complexes for the Zariski topology. The descent condition essentially only requires that when given two open sets $U_1$ and $U_2$, the induced square
\begin{center}
\begin{tikzpicture}[scale = 1.5]
\node (A2) at (2,2) {$\Gamma(U_1; \Fc)$};
\node (B2) at (2,1) {$\Gamma(U_1 \cap U_2; \Fc)$};
\node (A1) at (0,2) {$\Gamma(U_1 \cup U_2; \Fc)$};
\node (B1) at (0,1) {$\Gamma(U_2; \Fc)$};
\path[every node/.style={font=\sffamily\small}]
(A1) edge[->] (A2)
(B1) edge[->] (B2)
(A1) edge[->] (B1)
(A2) edge[->] (B2)
;
\end{tikzpicture}
\end{center}
is homotopy Cartesian. More generally, given a finite open cover $U_1,...,U_n$, the sections over the union should be a homotopy limit of a diagram indexed by all the different intersections of $U_i$.

It seems to be a well known fact that such homotopy limits can be computed as totalization of a double complex given by the diagram, but the original reference seems to be hard to find. A proof is given in \cite{AJS}, and hints of it are given in sections of \cite{Dug} discussing spectral sequences associated to homotopy limits and colimits. In this case the double complex is the level-wise Čech complex 
\begin{center}
\begin{tikzpicture}[scale = 1.5]
\node (A0) at (0,3) {$\vdots$};
\node (A1) at (0,2) {$\prod_{i}\Gamma(U_i; \Fc)_{d}$};
\node (A2) at (0,1) {$\prod_{i}\Gamma(U_i; \Fc)_{d-1}$};
\node (A3) at (0,0) {$\vdots$};

\node (B0) at (3,3) {$\vdots$};
\node (B1) at (3,2) {$\prod_{i < j}\Gamma(U_{ij}; \Fc)_{d}$};
\node (B2) at (3,1) {$\prod_{i < j}\Gamma(U_{ij}; \Fc)_{d-1}$};
\node (B3) at (3,0) {$\vdots$};

\node (C0) at (6,3) {$\vdots$};
\node (C1) at (6,2) {$\prod_{i < j < k}\Gamma(U_{ijk}; \Fc)_{d}$};
\node (C2) at (6,1) {$\prod_{i < j < k}\Gamma(U_{ijk}; \Fc)_{d-1}$};
\node (C3) at (6,0) {$\vdots$};

\node (D1) at (8,2) {$\cdots$};
\node (D2) at (8,1) {$\cdots$};

\path[every node/.style={font=\sffamily\small}]
(A0) edge[->] (A1)
(A1) edge[->] (A2)
(A2) edge[->] (A3)
(B0) edge[->] (B1)
(B1) edge[->] (B2)
(B2) edge[->] (B3)
(C0) edge[->] (C1)
(C1) edge[->] (C2)
(C2) edge[->] (C3)
(A1) edge[->] (B1)
(B1) edge[->] (C1)
(C1) edge[->] (D1)
(A2) edge[->] (B2)
(B2) edge[->] (C2)
(C2) edge[->] (D2)
;
\end{tikzpicture}
\end{center}
where $U_{i_1 i_2 \cdots i_k}$ is shorthand for $\cap_j U_{i_j}$. The vertical differentials are just product of the differentials of chain complexes, and the vertical differentials are the differentials of the Čech complex.

In our case of interest $\Fc$ will be coherent, so especially it will be eventually connective. Moreover the cover will be finite, so the entries will be eventually trivial in the vertical direction. Consider the spectral sequence of double complex where we take the first differential to be the vertical differential. The open sets $U_i$ were assumed to be affine and if we moreover make the extra assumption that all the intersections are affine as well (e.g. $X$ is separated), then in $E_1$ page we will have exactly the sections of the homotopy sheaves $\pi_i \Fc$ on the various intersections of $U_i$. After that, the $E_2$-page will consists of the cohomology groups of the homotopy sheaves. We may summarize this as a proposition.

\begin{prop}
Suppose we have a derived scheme $X$, a cover $(U_i)$ and a coherent sheaf $\Fc$ as above. Then there is a spectral sequence
\begin{equation*}
E_2^{p,q} = H^p(\pi_q \Fc) \implies \pi_{q-p}(\Gamma(X; \Fc))
\end{equation*}
whose differentials have the form $d_r: E_r^{p,q} \to E_r^{p+r,q+r-1}$.
\end{prop} 

We end by remarking that running the spectral sequence the other way (starting with the vertical differentials) is sometimes very convenient when doing computations.

\subsection{Twisting Sheaves}

We are now ready to prove the main results in the section. They will be rather straightforward exercises in applying the spectral sequence appearing in the previous section. We begin with an important definition:

\begin{defn}\label{Ampleness}
Let $X$ be a derived scheme and $\Ls$ a line bundle on $X$. We say that $\Ls$ is \emph{ample} if the line bundle $\pi_0$ is ample on the truncation $tX$.
\end{defn}

Any quasi-projective derived scheme $X$ has an ample line bundle: indeed, choose a locally closed embedding $X \to \Pb^n$ and take the pullback $\Oc_X(1)$ of the line bundle $\Oc_{\Pb^n}(1)$ on $\Pb^n$. Now the line bundle $\pi_0(\Oc_X(1))$ coincides with the pullback of $\Oc_{\Pb^n}(1)$ along $tX \to \Pb^n$, and is therefore ample. In the sequel $\Oc(1)$ is any ample line bundle on $X$, and for a quasi-coherent sheaf $\Fc$, $\Fc(n)$ is its $n^{th}$ twist $\Fc \otimes^L \Oc_X(n)$. We note that $\pi_i(\Fc(n)) = \pi_i(\Fc)(n)$ as line bundles are flat (cf. Proof of Lemma 2.2.2.2. (2) in \cite{TV2} or \cite{Lur2} Proposition 7.2.2.13.). 

\begin{thm}\label{TwistingTheorem}
Let $\Fc$ be a coherent sheaf on a quasi-projective derived scheme $X$. For any $i \in \Zb$, and all $n \gg 0$ (depending on $i$), we have a natural isomorphism
\begin{equation*}
\pi_i(\Gamma(X; \Fc(n))) \cong \Gamma(tX, \pi_i(\Fc)(n)).
\end{equation*} 
\textbf{THIS IS TRUE ONLY WHEN $tX$ IS PROJECTIVE}. Please consult \cite{An2} for quasi-projectivity \& related stuff...
\end{thm}
\begin{proof}
We use the descent spectral sequence from the previous section. The $E_2$ page is given by $E^{p,q}_2 = H^p(\pi_q \Fc)$. To prove the claim, we need to know that the $(i,0)$-cell does not support a nontrivial differential, and that there is nothing left on the $(i+j,j)$-diagonal on the $E_\infty$-page. But this is easy. We have only finitely many cells that could potentially cause any problems, and they all consist of higher cohomologies of coherent sheaves. Hence, by twisting sufficiently many times, the entries in the problematic cells vanish, giving us the desired isomorphism.
\end{proof}

As a direct corollary, we obtain a result concerning global generation of strong sheaves.

\begin{cor}\label{GlobalGeneration}
Let $\Fc$ be a strong and coherent sheaf (e.g. a vector bundle) on a quasi-projective derived scheme $X$. For $n \gg 0$, the sheaf $\Fc(n)$ is globally generated.
\end{cor}
\begin{proof}
For $n \gg 0$ we have that $\pi_0(\Fc)(n)$ is globally generated and $\pi_0(\Gamma(X; \Fc(n))) \cong \Gamma(X; \pi_0(\Fc)(n))$, proving the claim.
\end{proof}
The above theorem plays an important part in \cite{An} when constructing Chern classes for algebraic cobordism $\Omega^*$.

\begin{rem}
For strong sheaves $\Fc$, the space of global sections coincides with the space of maps $\Oc_X \to \Fc$ by definition. However, something strange seems to be going on as the descent spectral sequence may give us non-trivial negative homotopy groups. The reason for this is that we are actually computing the global sections of a stack taking values in unbounded chain complexes (spectra) instead of connective chain complexes (spaces). The space of global sections can be recovered from the spectrum of global sections by truncating it at zero. Especially, the non-negative homotopy groups given by the spectral sequence really compute the homotopy groups of the space of global sections.
\end{rem}

For coherent sheaves with only finitely many nontrivial homotopy sheaves, a slightly stronger version of \fref{TwistingTheorem} is true.

\begin{cor}
Let $\Fc$ be a coherent on a quasi-projective derived scheme $X$, and assume $\Fc$ has only finitely many nontrivial homotopy groups. Now for all $n \gg 0$, we have for all $i$ a natural isomorphism
\begin{equation*}
\pi_i(\Gamma(X; \Fc(n))) \cong \Gamma(tX, \pi_i(\Fc)(n)).
\end{equation*} 
\end{cor}
\begin{proof}
The proof is essentially the same as the proof of \fref{TwistingTheorem}.
\end{proof}

Moreover, under certain hypotheses, being quasi-projective is equivalent to admitting an ample line bundle.

\begin{thm}\label{AmpleMeansQProj}
Let $X$ be a derived $k$-scheme whose truncation is of finite type, and suppose $k$ is of characteristic 0. If $X$ admits an ample line bundle, then it is quasi-projective.
\end{thm}
\begin{proof}
Let $\Ls$ be the ample line bundle. By the earlier results of this section, for $n \gg 0$ the tensor power $\Ls^{\otimes n}$ is globally generated, $\pi_0(\Gamma(X; \Ls^{\otimes n})) \cong \Gamma(tX; \pi_0(\Ls)^{\otimes n})$ and $\pi_0(\Ls)^{\otimes n}$ is very ample. Pick any sections $s_0, ..., s_n$ of $\pi_0(\Ls)^{\otimes n}$ determining a locally closed embedding $tX \hookrightarrow \Pb^n$, and lift them to sections $\tilde{s}_0,...,\tilde{s}_n$ of $\Ls^{\otimes n}$. In characteristic $0$ it is known (see \cite{Lur3} Section 19.2.6) that this is equivalent to a data of a map $X \to \Pb^n$, and the truncation of this map is the original locally closed embedding. It follows that $X \to \Pb^n$ is a locally closed embedding and $X$ is quasi-projective.
\end{proof}

\subsection{Relative Ampleness}

It is often useful to have at hand a relative notion of ampleness, and to understand how it is related to a morphism being quasi-projective. The purpose of this subsection is to record straightforward generalizations of some classical facts to derived algebraic geometry. We begin with a definition:

\begin{defn}
Let $f: X \to Y$ be a morphism of derived $k$-schemes. A line bundle $\Li$ on $X$ is \emph{relatively ample over $Y$} (or \emph{$f$-ample}) if for every affine open $\Spec (A) \subset Y$, the line bundle $\Li \vert_{f^{-1} \Spec (A)}$ is ample in the sense of definition \fref{Ampleness}.
\end{defn}

Comparing the above definition to the classical case, we obtain:

\begin{prop}
Let $f: X \to Y$ be a morphism of derived $k$-schemes, and $\Li$ a line bundle on $X$. The following conditions are equivalent:
\begin{enumerate}
\item $\Li$ is $f$-ample;
\item the truncation $\pi_0(\Li)$ on $tX$ is ample over $tY$;
\item there exists an affine open cover $(\Spec(A_i))_{i \in I}$ of $Y$ so that the line bundles $\Li \vert_{f^{-1} \Spec(A_i)}$ are ample for all $i \in I$.
\end{enumerate}
\end{prop}
\begin{proof}
The equivalence of the conditions 1 and 2 is a direct consequence of the definition \fref{Ampleness}. On the other hand, as the equivalence of the conditions 1 and 3 is known for classical schemes, we obtain the analogous equivalence for derived schemes by checking on the truncation.
\end{proof}

Moreover, as relative ampleness is stable under classical pullbacks, we immediately obtain the following result.

\begin{prop}[Ampleness is stable under homotopy pullbacks]
Let $f: X \to Y$ and $g: Y' \to Y$ be morphisms of derived $k$-schemes, and let $\Li$ be an $f$-ample line bundle on $X$. Consider the homotopy Cartesian square
$$
\CD
X' @>{f'}>> Y'   \\
@V{g'}VV @V{g}VV \\
X @>{f}>> Y
\endCD
$$
Now $g'^* \Li$ is $f'$-ample.
\end{prop}

The main result of this subsection is the following:

\begin{thm}\label{RelAmpleMeansQProj}
Let $f: X \to Y$ be a morphism of derived $k$-schemes with $Y$ quasi-projective. Now the morphism $f$ is quasi-projectice if and only if $X$ admits an $f$-ample line bundle. 
\end{thm}

Of course, one of the directions is elementary: if $f$ factors through a closed embedding $i: X \hookrightarrow U \times Y$, then the line bundle $i^* \OO_U(1)$ is $f$-ample (this can be checked on the truncation). The other direction is taken care of by the following lemma:

\begin{lem}
Let $Y$ be quasi-projective, $f: X \to Y$ a morphism of derived $k$-schemes, and let $\Li$ on $X$ be ample over $Y$. Now $f$ factors as
$$X \stackrel i \hookrightarrow U \times Y \stackrel \pi \to Y,$$
where $i$ is a closed embedding, $\pi$ is the natural projection and $U$ is an open subscheme of a projective scheme $\Proj^n$.
\end{lem}
\begin{proof}
Let $\Mc$ be an ample line bundle on $Y$. As the truncation $\pi_0(\Li \otimes f^* \Mc)$ on $tX$ is known to be ample, we conclude that $\Li \otimes f^* \Mc$ is ample. Hence, by \fref{AmpleMeansQProj}, we can conclude that for $j$ large enough there are enough global sections of $\Li^{\otimes j} \otimes \Mc^{\otimes j}$ to determine a locally closed embedding $i: X \hookrightarrow \Proj^n$. Restricting the target of $i$ to $U$ so that $X \hookrightarrow U$ is a closed embedding, we obtain the desired factorization
$$X \hookrightarrow U \times Y \to Y$$
of $f$. This finishes the proof of Theorem \fref{RelAmpleMeansQProj}.
\end{proof}

\section{$K^0$ of Derived Schemes}

Recall how the $K$-theory of a derived scheme is defined in \cite{KST}: start with the stable $\infty$-category $\Perf(X)$ of the perfect sheaves on $X$, and use the universal construction of \cite{BGT} to arrive at the corresponding non-connective $K$-theory spectrum $K(X) := K(\Perf(X))$. We denote the zeroth homotopy group of this spectrum by $K^0(X)$. It is known that it is the free Abelian group on the equivalence classes of perfect sheaves modulo the relations
\begin{equation*}
[\Gc] = [\Fc] + [\Hc]
\end{equation*}
coming from homotopy cofibre sequences $\Fc \to \Gc \to \Hc$. Indeed, this holds for the corresponding Waldhausen $K$-theory, and by Corollary of \cite{BGT}, the connective spectrum constructed in \emph{loc. cit.} agrees with the Waldhausen $K$-theory spectrum. Moreover, by \emph{loc. cit.} Remark 9.33 the zeroth homotopy groups of the connective and non-connective spectrum agree as $\Perf(X)$ is idempotent complete. 

The main purpose of this section is to give an alternative description of $K^0(X)$ for a quasi-projective derived scheme $X$. Namely $K^0(X)$ is the free Abelian group on equivalence classes of vector bundles of $X$ modulo relations
\begin{equation*}
[E] = [E'] + [\overline{E}]
\end{equation*} 
coming from short exact sequences $E' \to E \to \overline{E}$ in the sense of Definition \fref{StrongExact}. This description is used in \cite{An} to establish a connection between $K^0$ and algebraic cobordism $\Omega^*$ for all quasi-projective derived schemes $X$. Recall that for quasi-projective derived schemes $X$ the perfect sheaves coincide with the compact objects in $\QCoh(X)$.

\subsubsection*{Resolving Perfect Sheaves}

We recall the following standard definition:

\begin{defn}
We say that a quasi-coherent sheaf $\Fc$ on a derived scheme $X$ has \emph{Tor-amplitude} $\leq n$ if for all discrete quasi-coherent sheaves $\Gc$, the homotopy sheaves of $\Fc \otimes^L \Gc$ are concentrated in degrees $\leq n$. (cf. \cite{Lur2} Definition 7.2.4.21.)
\end{defn}

Recall also the following standard result.

\begin{prop}
A perfect sheaf has finite Tor-amplitude.
\end{prop}
\begin{proof}
See  \cite{Lur2} Proposition 7.2.4.23. (4).
\end{proof}

We are also going to use the following
\begin{lem}
Let $\Fc$ be a connective perfect sheaf on a Noetherian derived scheme. Now $\Fc$ is vector bundle if and only if it has Tor-amplitude $\leq 0$.
\end{lem}
\begin{proof}
The question is clearly local, so we are reduced to the corresponding question concerning modules. By \cite{Lur2} Remark 7.2.4.22. $\Fc$ is flat if and only if it has Tor amplitude $\leq 0$ and by Lemma 7.2.2.18. of \emph{loc. cit.}, a flat $A$-module $M$ is projective if and only if the flat $\pi_0(A)$-module $\pi_0(M)$ is projective (of course, projective implies flat). Due to our Noetherian assumption, this is equivalent to requiring $\pi_0$ to be finitely generated, which is taken care of by \emph{loc. cit.} Propositions 7.2.4.11. and 7.2.4.17. stating that all perfect modules are almost perfect and that almost perfect module has finitely presented homotopy.
\end{proof}

Suppose that $\Fc$ is a connective perfect sheaf on quasi-projective derived scheme $X$, By the results of the previous section, for large enough $n$, we have a morphism $\Oc^{\oplus m}_X \to \Fc(n)$ inducing surjection on $\pi_0$. Hence, given a connective perfect sheaf $\Fc$ on $X$, we can always find a morphism $F \to \Fc$ from a vector bundle $F$ so that the induced map $\pi_0(F) \to \pi_0(\Fc)$ is an epimorphism. The homotopy fibre $\Kc$ of such a morphism is connective and perfect. Moreover, if $\Fc$ had Tor-amplitude $\leq n$, where $n$ is positive, then $\Kc$ has Tor-amplitude $\leq n-1$ (cf. \cite{Lur2} Remark 7.2.4.24.)

We can continue the process above so that we end up with a resolution
\begin{equation*}
F_d \to \cdots \to F_1 \to F_0 \to \Fc
\end{equation*}
If we denote by $\Kc_d$ the vector bundle $F_d$, and by $\Kc_{i-1}$ the homotopy cofibre of $\Kc_{i} \to F_{i-1}$, then the sequence is characterized by the fact that $\Kc_1 \to F_0 \to \Fc$ is a homotopy cofibre sequence. As shifting corresponds to multiplying with $-1$ in $K^0(X)$, we have shown that
\begin{equation*}
[\Fc] = \sum_{i=0}^d (-1)^i[F_i] \in K^0(X)
\end{equation*}
and therefore we can conclude that the classes of vector bundles generate the $K^0$ of a quasi-projective derived scheme.

\subsubsection*{Presentation of $K^0$}

Consider now the Abelian group $K^0_\mathrm{vect}(X)$ generated by the equivalence classes $[E]$ of vector bundles on $X$ modulo relations  
\begin{equation*}
[E] = [E'] + [\overline{E}]
\end{equation*}
coming from exact sequences $E' \to E \to \overline{E}$ in the sense of \fref{StrongExact}. By Proposition \fref{ExactVsCofibre} we have a map $K^0_\mathrm{vect}(X) \to K^0(X)$, and by the previous paragraph this morphism is surjective. Our next task is to prove the main theorem of this section:

\begin{thm}\label{KTheoryTheorem}
The map
\begin{equation*}
K^0_\mathrm{vect}(X) \to K^0(X)
\end{equation*}
defined above is an isomorphism.
\end{thm}

We want to show that the above map has an inverse, which we define by sending a perfect sheaf $\Fc$ to the $K^0_\mathrm{vect}$-class of its resolution. We need a sequence of lemmas to show that this is well defined. Denote by $\Zc^0(X)$ the free Abelian group on equivalence classes of perfect sheaves.

\begin{lem}
The map $\Zc^0(X) \to K^0_\mathrm{vect}(X)$ defined by sending
\begin{equation*}
[\Fc] \mapsto \sum_{i=0}^d (-1)^i[F_i] 
\end{equation*}
does not depend on the chosen resolution of $\Fc$.
\end{lem}
\begin{proof}
If $\Fc$ is already a vector bundle, then any resolution gives, by construction, a long exact sequence
\begin{equation*}
F_r \to \cdots \to F_1 \to F_0 \to \Fc 
\end{equation*}
of vector bundles. Hence a standard argument shows that $[\Fc] = \sum_i (-1)^i [F_i]$ in $K^0_\mathrm{vect}(X)$, and thus the image is well defined. Suppose then $\Fc$ has tor amplitude $\leq d+1$, and consider two beginnings of a resolution $F_0 \to \Fc$ and $F'_0 \to \Fc$. Consider the case when the vector bundle $F_0$ surjects onto $F'_0$ over $\Fc$ so that we obtain the following diagram
\begin{center}
\begin{tikzpicture}[scale = 1.5]
\node (C2) at (1,3) {$0$};
\node (C1) at (0,3) {$K$};
\node (C0) at (-1,3) {$K$};
\node (A2) at (1,2) {$\Fc$};
\node (B2) at (1,1) {$\Fc$};
\node (A1) at (0,2) {$F_0$};
\node (B1) at (0,1) {$F'_0$};
\node (A0) at (-1,2) {$\Kc$};
\node (B0) at (-1,1) {$\Kc'$};
\path[every node/.style={font=\sffamily\small}]
(C0) edge[->] (A0)
(C1) edge[->] (A1)
(C2) edge[->] (A2)
(C0) edge[->] (C1)
(C1) edge[->] (C2)
(A0) edge[->] (A1)
(B0) edge[->] (B1)
(A1) edge[->] (A2)
(B1) edge[->] (B2)
(A0) edge[->] (B0)
(A1) edge[->] (B1)
(A2) edge[->] (B2)
;
\end{tikzpicture}
\end{center}
whose rows and columns are all homotopy cofibre sequences. Now $K$ is a vector bundle as the homotopy fibre of a surjective map of vector bundles. By induction on Tor-amplitude the perfect sheaves $\Kc$ and $\Kc'$ have well defined images in $K^0_\mathrm{vect}$, and by Lemma \fref{TechnicalLemma} the images differ by $[K]$, which is exactly the difference of $F_0$ and $F'_0$. This shows that the $K^0_\mathrm{vect}$-class of a resolution does not depend on the choice of the first vector bundle, at least in the case where one choice surjects onto the other. As $F_0 \oplus F_0'$ dominates both $F_0$ and $F_0'$, we are done.
\end{proof}

\begin{lem}\label{TechnicalLemma}
Let us have a homotopy cofibre sequence $K \to \Fc \to \Fc'$ of connective perfect sheaves, and suppose $K$ is a vector bundle. Choose a resolution
\begin{equation*}
\Kc_{d-1} \to F_{d-1} \to \cdots \to F_0 \oplus K \to \Fc.
\end{equation*}
Now the induced sequence
\begin{equation*}
\Kc_{d-1} \to F_{d-1} \to \cdots \to F_0 \to \Fc'
\end{equation*}
is a resolution of $\Fc'$.
\end{lem}
\begin{proof}
Indeed, consider the diagram
\begin{center}
\begin{tikzpicture}[scale = 1.5]
\node (C2) at (1.2,3) {$K$};
\node (C1) at (0,3) {$K$};
\node (C0) at (-1.2,3) {$0$};
\node (A2) at (1.2,2) {$\Fc$};
\node (B2) at (1.2,1) {$\Fc'$};
\node (A1) at (0,2) {$F_0 \oplus K$};
\node (B1) at (0,1) {$F_0$};
\node (A0) at (-1.2,2) {$\Kc_0$};
\node (B0) at (-1.2,1) {$\Kc_0$};
\path[every node/.style={font=\sffamily\small}]
(C0) edge[->] (A0)
(C1) edge[->] (A1)
(C2) edge[->] (A2)
(C0) edge[->] (C1)
(C1) edge[->] (C2)
(A0) edge[->] (A1)
(B0) edge[->] (B1)
(A1) edge[->] (A2)
(B1) edge[->] (B2)
(A0) edge[->] (B0)
(A1) edge[->] (B1)
(A2) edge[->] (B2)
;
\end{tikzpicture}
\end{center}
where the top and the middle rows are homotopy cofibre sequences. The claim now follows from the fact that the bottom row is also a homotopy cofibre sequence.
\end{proof}

We now know that we have a well defined map from the equivalence classes of perfect sheaves to the Grothendieck group of vector bundles. The proof of \fref{KTheoryTheorem} is finished by the following lemma.

\begin{lem}
The map $\Zc^0(X) \to K^0_\mathrm{vect}(X)$ defined above descends to a map $K^0(X) \to K^0_\mathrm{vect}(X)$.
\end{lem}
\begin{proof}
We need to show that the map respects the relations given by cofibre sequences. Suppose we have a map $\Fc \to \Gc$ of connective perfect sheaves. Choose a vector bundle $F_0 \to \Fc$ inducing an epimorphism on $\pi_0$, and a vector bundle $F'_0$, so that $F_0 \oplus F'_0 \to \Gc$ induces an epimorphism on $\pi_0$. We now have the following diagram
\begin{center}
\begin{tikzpicture}[scale = 1.5]
\node (C2) at (1.2,3) {$\Kc''$};
\node (C1) at (0,3) {$\Kc'$};
\node (C0) at (-1.2,3) {$\Kc$};
\node (A2) at (1.2,2) {$F'_0$};
\node (B2) at (1.2,1) {$\Hc$};
\node (A1) at (0,2) {$F_0 \oplus F'_0$};
\node (B1) at (0,1) {$\Gc$};
\node (A0) at (-1.2,2) {$F_0$};
\node (B0) at (-1.2,1) {$\Fc$};
\path[every node/.style={font=\sffamily\small}]
(C0) edge[->] (A0)
(C1) edge[->] (A1)
(C2) edge[->] (A2)
(C0) edge[->] (C1)
(C1) edge[->] (C2)
(A0) edge[->] (A1)
(B0) edge[->] (B1)
(A1) edge[->] (A2)
(B1) edge[->] (B2)
(A0) edge[->] (B0)
(A1) edge[->] (B1)
(A2) edge[->] (B2)
;
\end{tikzpicture}
\end{center}
where all the rows and columns are homotopy cofibre sequences. Moreover $F'_0 \to \Hc$ induces a surjection on $\pi_0$, as can be readily seen from the fibration sequence on homotopy groups. 

Hence, if we have a homotopy cofibre sequence $\Fc \to \Gc \to \Hc$ of perfect sheaves of Tor-amplitude $\leq n + 1$, then the above observation reduces the question to cofibre sequence of perfect sheaves having Tor-amplitude $\leq n$. Recalling that a triangle of vector bundles is a cofibre sequence if and only if it is an exact sequence (Proposition \fref{ExactVsCofibre}), we see that the map $\Zc^0(X) \to K^0_\mathrm{vect}(X)$ respects cofibre sequence-relation, and therefore we obtain a map $K^0(X) \to K^0_\mathrm{vect}(X)$. This finishes the proof of \fref{KTheoryTheorem}.
\end{proof}

\end{document}